\documentclass{amsart}
\usepackage{longtable}
\usepackage{amssymb}
\usepackage{color}
\usepackage{graphicx}
\theoremstyle{plain}
\newtheorem{theorem}{Theorem}[section]
\newtheorem{corollary}[theorem]{Corollary}

\newtheorem{lemma}[theorem]{Lemma}
\numberwithin{equation}{section}

\makeatletter
\@namedef{subjclassname@2020}{%
  \textup{2020} Mathematics Subject Classification}
\makeatother

\begin{document}

\title[Hankel Determinants of Bernoulli and Euler Polynomials]{Hankel Determinants of sequences related to Bernoulli and Euler Polynomials}

\author{Karl Dilcher}
\address{Department of Mathematics and Statistics\\
         Dalhousie University\\
         Halifax, Nova Scotia, B3H 4R2, Canada}
\email{dilcher@mathstat.dal.ca}

\author[Lin Jiu]{Lin Jiu\textsuperscript{*}}
\thanks{*corresponding author}
\address{Department of Mathematics and Statistics\\
         Dalhousie University\\
         Halifax, Nova Scotia, B3H 4R2, Canada}
\email{lin.jiu@dal.ca}

\thanks{Research supported in part by the Natural Sciences and Engineering
        Research Council of Canada, Grant \# 145628481}

\date{}

\setcounter{equation}{0}

\begin{abstract}
We evaluate the Hankel determinants of various sequences related to Bernoulli
and Euler numbers and special values of the corresponding polynomials. Some of
these results arise as special cases of Hankel determinants of certain sums 
and differences of Bernoulli and Euler polynomials, while others are 
consequences of a method that uses the derivatives of Bernoulli and Euler
polynomials. We also obtain Hankel determinants for sequences of sums and differences 
of powers and for generalized Bernoulli polynomials belonging to certain Dirichlet
characters with small conductors. Finally, we collect and organize Hankel
determinant identities for numerous sequences, both new and known, containing
Bernoulli and Euler numbers and polynomials.
\end{abstract}

\maketitle

\section{Introduction}

The Bernoulli numbers $B_n$ and polynomials $B_n(x)$ are usually defined by 
the generating functions
\begin{equation}\label{1.1}
\frac{t}{e^t-1} = \sum_{n=0}^\infty B_n\frac{t^n}{n!}\qquad\hbox{and}\qquad
\frac{te^{xt}}{e^t-1} = \sum_{n=0}^\infty B_n(x)\frac{t^n}{n!},
\end{equation}
and the related  Euler numbers $E_n$ and polynomials $E_n(x)$ can be defined by 
\begin{equation}\label{1.2}
\frac{2}{e^t+e^{-t}} = \sum_{n=0}^\infty E_n\frac{t^n}{n!}\qquad\hbox{and}\qquad
\frac{2e^{xt}}{e^t+1} = \sum_{n=0}^\infty E_n(x)\frac{t^n}{n!}.
\end{equation}
These four sequences are 
among the most important special number and polynomial sequences in mathematics,
with numerous applications in number theory, combinatorics, numerical analysis,
and other areas. The first few elements of these sequences are listed in
Table~1, and their most important properties can be found, e.g., in 
\cite[Ch.~24]{DLMF}.

This paper will be concerned with Hankel determinants of sequences related to
Bernoulli and Euler numbers and special values of the corresponding polynomials.
A {\it Hankel matrix} or {\it persymmetric matrix} is a symmetric matrix which 
has constant entries along its antidiagonals; in other words, it is of the form
\begin{equation}\label{1.3}
\big(c_{i+j}\big)_{0\leq i,j\leq n}
=\begin{pmatrix}
c_{0} & c_{1} & c_{2} & \cdots & c_{n}\\
c_{1} & c_{2} & c_{3} & \cdots & c_{n+1}\\
c_{2} & c_{3} & c_{4} & \cdots & c_{n+2}\\
\vdots & \vdots & \vdots & \ddots & \vdots\\
c_{n} & c_{n+1} & c_{n+2} & \cdots & c_{2n}
\end{pmatrix}.
\end{equation}
A {\it Hankel determinant} is then the determinant of a Hankel matrix. 
Furthermore, given a sequence ${\bf c}=(c_0,c_1,\ldots)$ of numbers or 
polynomials, we define the $n$th Hankel determinant of ${\bf c}$ to be 
\begin{equation}\label{1.3a}
H_n({\bf c}) = H_n(c_k) = \det_{0\leq i,j\leq n}\big(c_{i+j}\big).
\end{equation}
If we use the second notation, $H_n(c_k)$, it is always assumed that the 
sequence begins with $k=0$; it should be noted that the value of the Hankel
determinant depends on this in an essential way. We shall return to this
issue in Section~6.

The Hankel determinants of a sequence are closely related to classical 
orthogonal polynomials; see, e.g., \cite[Ch.~2]{Is}. In fact, many evaluations
of Hankel determinants come from this connection and a related connection with
continued fractions. All this has been well-studied; see, e.g., the very
extensive treatments in \cite{Kr, Kr2, Mi}, and the numerous references
provided there.

Using these connections with orthogonal polynomials and continued fractions,
the current authors recently derived some apparently novel evaluations of 
Hankel determinants of certain subsequences of Bernoulli and Euler polynomials
\cite{DJ}. It is the purpose of this paper to use the results in \cite{DJ},
combined with some other results, to obtain new evaluations of Hankel
determinants of various sequences related to Bernoulli and Euler numbers and
polynomials.

To put this in perspective and provide an introductory example, we quote the
following well-known result for the Euler numbers $E_n$, due to Al-Salam and 
Carlitz \cite[Eq.~(4.2)]{AC}, namely
\begin{equation}\label{1.4}
H_n(E_k) = (-1)^{\binom{n+1}{2}}\prod_{\ell=1}^n\ell!^2\qquad
(n\geq 0).
\end{equation}
The closely related sequence $(kE_{k-1})$ turns out to be quite different.
As a corollary of one of the main results in this paper we obtain
\begin{equation}\label{1.5}
H_{2m+1}\big(kE_{k-1}\big) 
= (-1)^{m+1}2^{4m(m+1)}\prod_{\ell=1}^m\ell!^{8}\qquad(m\geq 0),
\end{equation}
with $H_{2m}(kE_{k-1})=0$.

This paper is structured as follows. In Section~2 we quote some 
identities and known results that will be used later in the paper. Section~3
deals with Hankel determinants of sums and differences of Bernoulli polynomials
with the same index, along with various consequences, and in Section~4 we 
derive analogous results for Euler polynomials. In Section~5 we introduce a
method based on the derivatives of Bernoulli and Euler polynomials and obtain
several more Hankel determinant evaluations as corollaries. We also recall some
necessary facts on orthogonal polynomials in Section~5, and apply this 
derivative method to a shifted sequence in Section~6. Finally, in Section~7,
we collect and organize Hankel determinant identities for numerous sequences
containing Bernoulli and Euler numbers and polynomials.

\section{Some known results}

We begin this section with a few general properties of Bernoulli and 
Euler polynomials and of Hankel determinants that will 
be required in later sections.
We begin with two identities that connect Bernoulli and Euler numbers with their
polynomial analogues:
\begin{equation}\label{2.1}
B_n(x) = \sum_{j=0}^n\binom{n}{j}B_j x^{n-j},\qquad
E_n(x) = \sum_{j=0}^n\binom{n}{j}
\frac{E_j}{2^j}\big(x-\tfrac{1}{2}\big)^{n-j}.
\end{equation}
These identities follow easily from \eqref{1.1}, resp.\ \eqref{1.2}. The 
Bernoulli and Euler polynomials are also connected to each other through
\begin{equation}\label{2.2}
E_{n-1}(x) = \frac{2^n}{n}\left(B_n(\tfrac{x+1}{2})-B_n(\tfrac{x}{2})\right)
\end{equation}
(see, e.g., \cite[Eq.~24.4.23]{DLMF}), with the related identities
\begin{equation}\label{2.3}
(2n+1)E_{2n} = 2^{4n+2}B_{2n+1}(\tfrac{3}{4}),\quad
(n+1)E_n(1) = 2(2^{n+1}-1)B_{n+1}\quad(n\geq 1),
\end{equation}
which follow easily from \cite[Eq.~24.4.31, 24.4.26]{DLMF}.
Other important properties are the pair of reflection formulas
\begin{equation}\label{2.4}
B_{n}(1-x)=(-1)^{n}B_{n}(x),\qquad E_{n}(1-x)=(-1)^{n}E_{n}(x);
\end{equation}
see, e.g., \cite[Eq.~24.4.3, 24.4.4]{DLMF}, and the zeros
\begin{equation}\label{2.4a}
B_{2k+1}(\tfrac{1}{2})=B_{2k+3}(0)=B_{2k+3}(1)=0,\qquad k=0,1,2,\ldots
\end{equation}
and similarly
\begin{equation}\label{2.4b}
E_{2k+1}(\tfrac{1}{2})=E_{2k+2}(0)=E_{2k+2}(1)=0,\qquad k=0,1,2,\ldots
\end{equation}
Most of these follow from the identities above; see also 
\cite[Sect,~24.4(vi)]{DLMF}.

\bigskip
\begin{center}
\renewcommand*{\arraystretch}{1.2}
\begin{tabular}{|r||r|r|l|l|}
\hline
$n$ & $B_n$ & $E_n$ & $B_n(x)$ & $E_n(x)$\\
\hline
0 & 1 & 1 &  1 & 1 \\
1 & $-1/2$ & 0 & $x-\tfrac{1}{2}$ & $x-\tfrac{1}{2}$ \\
2 & $1/6$ & $-1$  & $x^2-x+\tfrac{1}{6}$ & $x^2-x$ \\
3 & 0 & 0 & $x^3-\tfrac{3}{2}x^2+\tfrac{1}{2}x$ & $x^3-\tfrac{3}{2}x^2+\tfrac{1}{4}$ \\
4 & $-1/30$ & 5 & $x^4-2x^3+x^2-\tfrac{1}{30}$ & $x^4-2x^3+x$ \\
5 & 0 & 0 & $x^5-\tfrac{5}{2}x^4+\tfrac{5}{3}x^3-\tfrac{1}{6}x$ &
$x^5-\tfrac{5}{2}x^4+\tfrac{5}{2}x^2-\tfrac{1}{2}$ \\
6 & $1/42$ & $-61$ & $x^6-3x^5+\tfrac{5}{2}x^4-\tfrac{1}{2}x^2+\tfrac{1}{42}$ &
$x^6-3x^5+5x^3-3x$ \\
\hline
\end{tabular}

\medskip
{\bf Table~1}: $B_n, E_n, B_n(x)$ and $E_n(x)$ for $0\leq n\leq 6$.
\end{center}
\bigskip

Next we state some useful properties of Hankel determinants.

\begin{lemma}\label{lem:2.1}
Let $(c_0, c_1,\ldots)$ be a sequence and $x$ a variable or a complex number. 
Then for all $n\geq 0$ we have
\begin{equation}\label{2.5}
H_n(x^kc_k) = x^{n(n+1)}H_n(c_k),
\end{equation}
and if
\[
c_k(x) = \sum_{j=0}^k\binom{k}{j}c_jx^{k-j},
\]
then
\begin{equation}\label{2.6}
H_n(c_k(x)) = H_n(c_k).
\end{equation}
\end{lemma}

The identity \eqref{2.5} is easy to derive by factoring suitable powers of $x$
from the rows and columns of the Hankel determinant. \eqref{2.6} can be found,
with proof, in \cite{Ju}; it is also mentioned and used in various other 
publications, for instance in \cite[Lemma~15]{Kr}. Applying Lemma~\ref{lem:2.1}
to both identities in \eqref{2.1}, we obtain
\begin{align}
H_n(B_k(x)) &= H_n(B_k),\label{2.7}\\
H_n(E_k(x)) &= 2^{-n(n+1)}H_n(E_k).\label{2.8}
\end{align}
The Hankel determinant on right-hand side of \eqref{2.8} was already mentioned 
in \eqref{1.4}, while the right-hand side of \eqref{2.7} will be recalled in the
final section.

The next lemma is about determinants of ``checkerboard matrices", namely 
matrices in which every other entry vanishes. This result can be found in 
\cite{CK} as Lemmas~5 and~6, and covers more general matrices than just Hankel
matrices.

\begin{lemma}\label{lem:2.2}
Let $M=\left(M_{i,j}\right)_{0\leq i,j\leq n-1}$ be a matrix.
If $M_{i,j}=0$ whenever $i+j$ is odd, then
\begin{equation}\label{2.9}
\det_{0\leq i.j\leq n-1}(M_{i,j})
=\det_{0\leq i.j\leq\lfloor(n-1)/2\rfloor}(M_{2i,2j})
\cdot\det_{0\leq i.j\leq\lfloor(n-2)/2\rfloor}(M_{2i+1,2j+1}).
\end{equation}
If $M_{i,j}=0$ whenever $i+j$ is even, then for even $n$ we have 
%(noting that, here is $n\times n$, instead of $(n+1)\times(n+1)$) 
\begin{equation}\label{2.10}
\det_{0\leq i.j\leq n-1}(M_{i,j})=
(-1)^{n/2}\det_{0\leq i.j\leq\lfloor(n-1)/2\rfloor}(M_{2i+1,2j})
\cdot\det_{0\leq i.j\leq\lfloor(n-2)/2\rfloor}(M_{2i,2j+1}), 
\end{equation}
while for odd $n$ we have
\begin{equation}\label{2.11}
\det_{0\leq i.j\leq n-1}(M_{i,j}) = 0.
\end{equation}
\end{lemma}

We note that there is a small typographical error in Lemma~6 of \cite{CK}; 
\eqref{2.10} above shows the correct power of $(-1)$.
Lemma~\ref{lem:2.2} is best explained through some examples.

\medskip
\noindent
{\bf Example~1.} By \eqref{2.9} we have
\[
\det\begin{pmatrix}a & 0 & b & 0 & c\\
0 & {\bf d} & 0 & {\bf e} & 0\\
f & 0 & g & 0 & h\\
0 & {\bf i} & 0 & {\bf j} & 0\\
k & 0 & l & 0 & m
\end{pmatrix}=\det\begin{pmatrix}a & b & c\\
f & g & h\\
k & l & m
\end{pmatrix}\cdot\det\begin{pmatrix}{\bf d} & {\bf e}\\
{\bf i} & {\bf j}
\end{pmatrix},
\]
\[
\det\begin{pmatrix}a & 0 & b & 0\\
0 & {\bf d} & 0 & {\bf e}\\
f & 0 & g & 0\\
0 & {\bf i} & 0 & {\bf j}
\end{pmatrix}=\det\begin{pmatrix}a & b\\
f & g
\end{pmatrix}\cdot\det\begin{pmatrix}{\bf d} & {\bf e}\\
{\bf i} & {\bf j}
\end{pmatrix}.
\]

\bigskip
\noindent
{\bf Example~2.} By \eqref{2.10} and \eqref{2.11} we have
\[
\det\begin{pmatrix}0 & {\bf d} & 0 & {\bf e}\\
f & 0 & g & 0\\
0 & {\bf i} & 0 & {\bf j}\\
k & 0 & l & 0
\end{pmatrix}=\det\begin{pmatrix}f & g\\
k & l
\end{pmatrix}\cdot\det\begin{pmatrix}{\bf d} & {\bf e}\\
{\bf i} & {\bf j}
\end{pmatrix},
\]
and the smaller cases
\[
\det\begin{pmatrix}0 & {\bf d}\\
a & 0
\end{pmatrix}=-a\cdot{\bf d},
\qquad
\det\begin{pmatrix}0 & {\bf d} & 0\\
f & 0 & g\\
0 & {\bf i} & 0
\end{pmatrix}=0.
\]

We also recall a well-known property of determinants, which
is easy to prove. Let $M=\left(M_{i,j}\right)$ be an $n\times n$ matrix,
and $\lambda$ be a constant. Then
\begin{equation}\label{2.12}
\det(\lambda M_{i,j})
= \lambda^n\det(M_{i,j}).
\end{equation}
This property will be required in the following sections.

We conclude this section by quoting two of the main results from \cite{DJ}.

\begin{theorem}[\cite{DJ}, Theorem~1.1]\label{thm:2.3}
If $b_k=B_{2k+1}(\frac{x+1}{2})$, then for $n\geq 0$ we have
\begin{equation}\label{2.13}
H_n(b_k) = (-1)^{\binom{n+1}{2}}\left(\frac{x}{2}\right)^{n+1}\prod_{\ell=1}^{n}
\left(\frac{\ell^{4}(x^{2}-\ell^{2})}{4(2\ell+1)(2\ell-1)}\right)^{n+1-\ell}.
\end{equation}
\end{theorem}

\begin{theorem}[\cite{DJ}, Corollary~5.2]\label{thm:2.4}
Let $c_k^{(\nu)}=E_{2k+\nu}(\frac{x+1}{2})$ for $\nu=0, 1, 2$. Then for all
$n\geq 0$ we have
\begin{equation}\label{2.14}
H_n(c_k^{(\nu)})=(-1)^{\binom{n+1}{2}}E_{\nu}(\tfrac{x+1}{2})^{n+1}
\prod_{\ell=1}^n\left(\frac{\ell^{2}}{4}\big(x^{2}-(2\ell+\nu-1)^{2}\big)\right)^{n+1-\ell},
\end{equation}
or more explicitly,
\begin{align}
H_n(c_k^{(0)}) &= (-1)^{\binom{n+1}{2}}\prod_{\ell=1}^n
\left(\frac{\ell^2}{4}\big(x^2-(2\ell-1)^2\big)\right)^{n+1-\ell},\label{2.15}\\
H_n(c_k^{(1)}) &= (-1)^{\binom{n+1}{2}}\left(\frac{x}{2}\right)^{n+1}
\prod_{\ell=1}^n\left(\frac{\ell^2}{4}\big(x^2-(2\ell)^2\big)\right)^{n+1-\ell},\label{2.16}\\
H_n(c_k^{(2)}) &= (-1)^{\binom{n+1}{2}}\left(\frac{x^{2}-1}{4}\right)^{n+1}
\prod_{\ell=1}^{n}\left(\frac{\ell^{2}}{4}(x^{2}-(2\ell+1)^{2})\right)^{n+1-\ell}.\label{2.17}
\end{align}
\end{theorem}

\section{Sums and Differences of Bernoulli polynomials}

While the Hankel determinants of the Euler numbers $E_k$ and Euler polynomials
$E_k(x)$ are well-known (see \eqref{1.4} and \eqref{2.8}), this is not the 
case for $H_n(b_k)$,
where $b_k=k\cdot E_{k-1}(x)$ for $k\geq 1$, and $b_0=0$. In order to deal with
this case, one could try to use the identity \eqref{2.2} in the form 
\begin{equation}\label{6.1}
k\cdot E_{k-1}(x) = -2^k\left(B_k(\tfrac{x}{2})-B_k(\tfrac{x+1}{2})\right),
\end{equation}
It is the purpose of this section to
show that we can obtain meaningful results for much more general differences,
as well as sums, of Bernoulli polynomials than the right-hand side of 
\eqref{6.1}.

We fix integers $q\geq 1$ and $0\leq r<s$, and define
\begin{equation}\label{6.2}
b_k^{\pm}(q,r,s;x) := B_k(\tfrac{x+r}{q})\pm B_k(\tfrac{x+s}{q}),
\qquad k=0, 1, 2,\ldots
\end{equation}
First we show that, just as in the case of Bernoulli and Euler polynomials,
the Hankel determinants of the polynomials in \eqref{6.2} do not depend on $x$.
This will greatly simplify further work in this section.

\begin{lemma}\label{lem:6.1}
For any $n\geq 0$, $H_n\big(b_k^{\pm}(q,r,s;x)\big)$ is independent of $x$.
\end{lemma}

\begin{proof}
Using a well-known identity (see, e.g., \cite[Eq.~24.4.12]{DLMF}) for 
Bernoulli polynomials, we get with \eqref{6.2},
\begin{align*}
b_k^{\pm}(q,r,s;x) 
&= \sum_{j=0}^k\binom{k}{j}B_j(\tfrac{r}{q})\big(\tfrac{x}{q}\big)^{k-j}
\pm\sum_{j=0}^k\binom{k}{j}B_j(\tfrac{s}{q})\big(\tfrac{x}{q}\big)^{k-j}\\
&= \sum_{j=0}^k\binom{k}{j}\left(B_j(\tfrac{r}{q})
\pm B_j(\tfrac{s}{q})\right)\big(\tfrac{x}{q}\big)^{k-j} \\
&= \sum_{j=0}^k\binom{k}{j}b_k^{\pm}(q,r,s;0)\big(\tfrac{x}{q}\big)^{k-j}.
\end{align*}
By Lemma~\ref{lem:2.1} this means that for any $n$ the $n$th Hankel determinant
of the sum on the right is independent of $x$, which completes the proof.
\end{proof}

The next lemma shows how we can exploit Lemma~\ref{lem:6.1} by being able to
choose an appropriate value for $x$.

\begin{lemma}\label{lem:6.2}
For fixed integers $q\geq 1$ and $0\leq r<s$, we have
\begin{equation}\label{6.3}
b_k^{-}(q,r,s;\tfrac{q-r-s}{2}) = \begin{cases}
0 &\hbox{if $k$ is even},\\
2B_{2\mu+1}(\tfrac{q+r-s}{2q}) &\hbox{if}\;\;k=2\mu+1,
\end{cases}
\end{equation}
and
\begin{equation}\label{6.3a}
b_k^{+}(q,r,s;\tfrac{q-r-s}{2}) = \begin{cases}
2B_{2\mu}(\tfrac{q+r-s}{2q}) &\hbox{if}\;\;k=2\mu,\\
0 &\hbox{if $k$ is odd}.
\end{cases}
\end{equation}
\end{lemma}

\begin{proof}
With \eqref{6.2} we get
\begin{align*}
b_k^{\pm}(q,r,s;\tfrac{q-r-s}{2})
&= B_k(\tfrac{q+r-s}{2q})\pm B_k(\tfrac{q-r+s}{2q})
= B_k(\tfrac{q+r-s}{2q})\pm B_k(1-\tfrac{q+r-s}{2q}) \\
&= B_k(\tfrac{q+r-s}{2q})\pm (-1)^kB_k(\tfrac{q+r-s}{2q}),
\end{align*}
where we have used the reflection formula \eqref{2.4}. The
desired identities \eqref{6.3} and \eqref{6.3a} now follow immediately.
\end{proof}

The significance of Lemma~\ref{lem:6.2} lies in the fact that the Hankel
matrices of the sequence in \eqref{6.3} are ``checkerboard matrices", and
therefore Lemma~\ref{lem:2.2} applies. The following is the main result of this
section.

\begin{theorem}\label{thm:6.3}
Let $q\geq 1$ and $0\leq r<s$ be fixed integers and set 
$b_k:=b_k^{-}(q,r,s;x)$. Then for all integers $m\geq 0$ we have 
$H_{2m}(b_k)=0$ and
\begin{equation}\label{6.4}
H_{2m+1}(b_k)=(-1)^{m+1}\left(\frac{s-r}{q^{m+1}}\right)^{2m+2}\prod_{\ell=1}^m
\left(\frac{\ell^4\big((s-r)^2-(q\ell)^2\big)}{4(2\ell+1)(2\ell-1)}\right)^{2(m+1-\ell)}.
\end{equation}
\end{theorem}

\begin{proof}
By Lemmas~\ref{lem:6.1} and~\ref{lem:6.2}, we can fix $x=(q-r-s)/2$ for $b_k$ and apply the second part of Lemma~\ref{lem:2.2} with $M_{i,j}=b_{i+j}$. Then
\eqref{2.11} immediately gives the first statement. In order to prove 
\eqref{6.4}, we begin by using \eqref{2.10} with $n=2m+2$, which gives
\begin{equation}\label{6.5}
\det_{0\leq i,j\leq 2m+1}\big(b_{i+j}\big)
= (-1)^{m+1}\left(\det_{0\leq i,j\leq m}\big(b_{2(i+j)+1}\big)\right)^2
= (-1)^{m+1}H_m(b_{2\mu+1})^2.
\end{equation}
By \eqref{6.3}, and using \eqref{2.12}, we have
\begin{equation}\label{6.6}
H_m(b_{2\mu+1}) = 2^{m+1}H_m\big(B_{2\mu+1}(\tfrac{q+r-s}{2q})\big).
\end{equation}
Finally we use Theorem~\ref{thm:2.3} with $x=(r-s)/q$, which gives 
$(x+1)/2=(q+r-s)/2q$, as required. Combining this with \eqref{2.13}, \eqref{6.6}
and \eqref{6.5}, we readily obtain \eqref{6.4}.
\end{proof}

We note that there is no analogue to Theorem~\ref{thm:6.3} for
$b_k^{+}(q,r,s;x)$. The main reason for this is the absence of an identity such
as \eqref{2.13} for even-index Bernoulli polynomials. This fact is briefly
discussed in \cite[Ch.~4]{DJ}. 

We now consider a special case of Theorem~\ref{thm:6.3}. Returning to 
\eqref{6.1}, we get the following results.

\begin{corollary}\label{cor:6.4}
For all integers $m\geq 0$ we have 
$H_{2m}\big(kE_{k-1}(x)\big)=H_{2m}\big(kE_{k-1}\big)=0$, and
\begin{align}
H_{2m+1}\big(kE_{k-1}(x)\big)
&= (-1)^{m+1}\prod_{\ell=1}^m\ell!^{8},\label{6.7} \\
H_{2m+1}\big(kE_{k-1}\big)
&= (-1)^{m+1}2^{4m(m+1)}\prod_{\ell=1}^m\ell!^{8}.\label{6.8}
\end{align}
\end{corollary}

\begin{proof}
Comparing \eqref{6.1} with \eqref{6.2}, we see that
$kE_{k-1}(x)=-2^kb_k^{-}(2,0,1;x)$. If we use \eqref{2.12}, \eqref{2.5}, and 
\eqref{6.4}, we get \eqref{6.7} after some straightforward manipulations.

To obtain \eqref{6.8}, we use the well-known identity $E_n=2^nE_n(\frac{1}{2})$,
which is a special case of the second part of \eqref{2.1} with $x=1/2$. Then 
with Lemma~\ref{lem:2.1} we get
\begin{align*}
H_{2m+1}\big(kE_{k-1}\big)
&=H_{2m+1}\big(\tfrac{1}{2}2^kkE_{k-1}(\tfrac{1}{2})\big) \\
&=\big(\tfrac{1}{2}\big)^{2m+2}2^{(2m+2)(2m+1)}
H_{2m+1}\big(kE_{k-1}(\tfrac{1}{2})\big),
\end{align*}
and \eqref{6.8} follows immediately from \eqref{6.7}.
\end{proof}

Theorem~\ref{thm:6.3} can also be used to deal with a few special cases of
character analogues of Bernoulli numbers and polynomials. Let $\chi$ be a 
primitive character with conductor $q$. Then the generalized Bernoulli numbers
and polynomials belonging to $\chi$ are defined by
\begin{equation}\label{6.10}
\sum_{a=1}^q\frac{\chi(a)te^{at}}{e^{qt}-1} 
= \sum_{n=0}^\infty B_{n,\chi}\frac{t^n}{n!},\qquad
B_{n,\chi}(x) = \sum_{k=0}^n\binom{n}{k}B_{k,\chi}x^{n-k},
\end{equation}
so that $B_{n,\chi}(0)=B_{n,\chi}$ for all $n\geq 0$.
These objects contain both the Bernoulli and Euler numbers and polynomials as
special cases. Indeed,
\begin{equation}\label{6.10a}
B_n(x)=B_{n,\chi_{0}}(x-1),\qquad
E_{n}(x)=-\frac{2^{1-n}}{n+1}B_{n+1,\chi_{4}}(2x-1),
\end{equation}
where $\chi_{0}$ is the trivial character and $\chi_{4}$ is the unique (non-trivial)
character with conductor 4; see, e.g., \cite[Sect.~24.16(ii)]{DLMF}. 
It was the right-hand identity in \eqref{6.10a}, by the way, that led us to 
first consider Hankel determinants of $k\cdot E_{k-1}(x)$.

On the other hand, all
generalized Bernoulli polynomials can be written in terms of the ordinary
Bernoulli polynomials by way of the identity
\begin{equation}\label{6.11}
B_{n,\chi}(x) = q^{n-1}\sum_{a=1}^q\chi(a)B_n(\tfrac{a+x}{q}),
\end{equation}
which follows easily from comparing the generating functions in \eqref{1.1} and
\eqref{6.10}. 

There are just three primitive characters that have exactly two nonzero values
between $a=1$ and $a=q$, namely those with conductors $q=3$, 4, and 6. In all
cases we have $\chi_q(1)=1$, $\chi_q(q-1)=-1$, and 0 elsewhere. With 
\eqref{6.11} and \eqref{6.2} we then have
\begin{equation}\label{6.12}
B_{k,\chi_q} = q^{k-1}b_k^{-}(q,1,q-1;0)\qquad(q=3, 4, 6).
\end{equation}
With Theorem~\ref{thm:6.3} we then get the following result.

\begin{corollary}\label{cor:6.5}
For $q=3, 4, 6$ and for all integers $m\geq 0$ we have $H_{2m}(B_{k,\chi_q})=0$
and
\begin{align}
H_{2m+1}(B_{k,\chi_q})&=(-1)^{m+1}\left(q^{m-1}(q-2)\right)^{2m+2}\label{6.13}\\
&\quad\times\prod_{\ell=1}^m \left(\frac{\ell^4\big((q-2)^2-(q\ell)^2\big)}{4(2\ell+1)(2\ell-1)}\right)^{2(m+1-\ell)}.\nonumber
\end{align}
\end{corollary}

\begin{proof}
We note that by \eqref{6.12}, and using \eqref{2.5} and \eqref{2.12}, we have
\begin{align*}
H_{2m+1}(B_{k,\chi_q}) &= H_{2m+1}(q^{-1}q^kb_k^{-}(q,1,q-1;0)\\
&= q^{-(2m+2)}q^{(2m+1)(2m+2)}H_{2m+1}(b_k^{-}(q,1,q-1;0)).
\end{align*}
The desired results now follow directly from Theorem~\ref{thm:6.3}.
\end{proof}

Considering the numerator in the right-most fraction in \eqref{6.4}, we get the
following immediate consequence.

\begin{corollary}\label{cor:6.6}
If $q\mid s-r$, then $H_{2m+1}(b_k^{-}(q,r,s;x))=0$ for all $m\geq(s-r)/q$. 
On the other hand, if $q\nmid s-r$, then $H_{2m+1}(b_k^{-}(q,r,s;x))\neq 0$ 
for all $m\geq 0$.
\end{corollary}

We conclude this section with an application of Theorem~\ref{thm:6.3} for $q=1$,
which also illustrates Corollary~\ref{cor:6.6}. Using the well-known identity
\[
x^{k-1}+(x+1)^{k-1}+\cdots+(x+s-1)^{k-1}=\frac{1}{k}\left(B_k(x+s)-B_k(x)\right)
\]
for integers $s\geq 1$ (see, e.g., \cite[Eq.~24.4.9]{DLMF}) we have by
\eqref{6.2}, with $x=1$,
\begin{equation}\label{6.14}
k\left(1+2^{k-1}+\cdots+s^{k-1}\right) = -b_k^{-}(1,0,s;1).
\end{equation}
With \eqref{2.12} we see that the $-$ sign on the right is irrelevant; 
Theorem~\ref{thm:6.3} now implies the following result.

\begin{corollary}\label{cor:6.7}
Let $S_k(s)$ be the left-hand side of \eqref{6.14}. Then for all integers 
$m\geq 0$ we have $H_{2m}(S_k(s))=0$ and
\[
H_{2m+1}(S_k(s))=(-s^2)^{m+1}\prod_{\ell=1}^m
\left(\frac{\ell^4(s^2-\ell^2)}{4(2\ell+1)(2\ell-1)}\right)^{2(m+1-\ell)},
\]
and in particular, $H_{2m+1}(S_k(s))=0$ for $m\geq s$.
\end{corollary}

A similar result was earlier obtained by Al-Salam and Carlitz \cite[Eq.~(7.1)]{AC}.

\section{Sums and Differences of Euler polynomials}

In this section we present ``Euler analogues" to some of the results in the
previous section. In analogy to \eqref{6.2} we fix integers $q\geq 1$ and
$0\leq r<s$ and define
\begin{equation}\label{4.1}
e_k^{\pm}(q,r,s;x) := E_k(\tfrac{x+r}{q})\pm E_k(\tfrac{x+s}{q}),
\qquad k=0, 1, 2,\ldots
\end{equation}
Since the main identity used in the proof of Lemma~\ref{lem:6.1} also holds
for Euler polynomials (see \cite[Eq.~24.4.13]{DLMF}), we have

\begin{lemma}\label{lem:4.1}
For any $n\geq 0$, $H_n\big(e_k^{\pm}(q,r,s;x)\big)$ is independent of $x$.
\end{lemma}
Furthermore, since the reflection formulas in \eqref{2.4} are identical for
both the Bernoulli and Euler polynomials, the following lemma also carries
over from the Bernoulli case.

\begin{lemma}\label{lem:4.2}
For fixed integers $q\geq 1$ and $0\leq r<s$, we have
\begin{equation}\label{4.2}
e_k^{-}(q,r,s;\tfrac{q-r-s}{2}) = \begin{cases}
0 &\hbox{if $k$ is even},\\
2E_{2\mu+1}(\tfrac{q+r-s}{2q}) &\hbox{if}\;\;k=2\mu+1,
\end{cases}
\end{equation}
and
\begin{equation}\label{4.3}
e_k^{+}(q,r,s;\tfrac{q-r-s}{2}) = \begin{cases}
2E_{2\mu}(\tfrac{q+r-s}{2q}) &\hbox{if}\;\;k=2\mu,\\
0 &\hbox{if $k$ is odd}.
\end{cases}
\end{equation}
\end{lemma}

While in Theorem~\ref{thm:6.3} we only obtained a result in the ``$-$" case,
for Euler polynomials we get meaningful results in both cases.

\begin{theorem}\label{thm:4.3}
Let $q\geq 1$ and $0\leq r<s$ be fixed integers and set
$e_k^{\pm}:=e_k^{\pm}(q,r,s;x)$. Then for all integers $m\geq 0$ we have
$H_{2m}(e_k^{-})=0$ and
\begin{equation}\label{4.4}
H_{2m+1}(e_k^{-})=(-1)^{m+1}\left(\frac{s-r}{q}\right)^{2m+2}\prod_{\ell=1}^m
\left(\frac{\ell^2}{4}\left(\left(\frac{s-r}{q}\right)^2-(2\ell)^2\right)\right)^{2(m+1-\ell)}.
\end{equation}
Furthermore, we have
\begin{equation}\label{4.5}
H_{2m}(e_k^{+})=(-1)^{m}\frac{2^{2m+1}}{m!^2}\prod_{\ell=1}^m
\left(\frac{\ell^2}{4}\left(\left(\frac{s-r}{q}\right)^2-(2\ell-1)^2\right)\right)^{2(m+1-\ell)}
\end{equation}
and
\begin{equation}\label{4.6}
H_{2m+1}(e_k^{+})=\prod_{\ell=0}^m\frac{\ell!^4}{16^\ell}
\left(\left(\frac{s-r}{q}\right)^2-(2\ell+1)^2\right)^{2(m-\ell)+1}. 
\end{equation}
\end{theorem}

\begin{proof}
Fix $x=(q-r-s)/2$. The proof for $e_k^{-}$ is similar to that of Theorem~\ref{thm:6.3}. 
We use again the second part of Lemma~\ref{lem:2.2}, this time with 
$M_{i,j}=e_{i+j}^{-}$; then \eqref{2.11} shows that $H_{2m}(e_k^{-})=0$.
To prove \eqref{4.4}, we begin by using \eqref{2.10} with $n=2m+2$, which gives
\begin{equation}\label{4.7}
\det_{0\leq i,j\leq 2m+1}\big(e_{i+j}^{-}\big)
= (-1)^{m+1}\left(\det_{0\leq i,j\leq m}\big(e_{2(i+j)+1}^{-}\big)\right)^2
= (-1)^{m+1}H_m(e_{2\mu+1}^{-})^2.
\end{equation}
With \eqref{4.2} and \eqref{2.12} we have
\begin{equation}\label{4.8}
H_m(e_{2\mu+1}^{-}) = 2^{m+1}H_m\big(E_{2\mu+1}(\tfrac{q+r-s}{2q})\big).
\end{equation}
Finally we use \eqref{2.16} with $x=(r-s)/q$. Then with \eqref{4.8} and
\eqref{4.7} we immediately get \eqref{4.4}.

For $e_k^{+}$ we use the first part of Lemma~\ref{lem:2.2} and distinguish
between two cases. First, when $n=2m+1$, then by \eqref{2.9} we have
\[
\det_{0\leq i,j\leq 2m}\big(e_{i+j}^{+}\big)
= \det_{0\leq i,j\leq m}\big(e_{2(i+j)}^{+}\big)
\cdot\det_{0\leq i,j\leq m-1}\big(e_{2(i+j)+2}^{+}\big),
\]
and with \eqref{4.3} and \eqref{2.12} we get
\begin{equation}\label{4.9}
H_{2m}(e_{k}^{+}) = 2^{2m+1}H_m\big(E_{2\mu}(\tfrac{q+r-s}{2q})\big)
\cdot H_{m-1}\big(E_{2\mu+2}(\tfrac{q+r-s}{2q})\big).
\end{equation}
Then we substitute \eqref{2.15} with $n=m$ and \eqref{2.17} with $n=m-1$ into
\eqref{4.9}, both with $x=(r-s)/q$. After some straightforward manipulations
we finally obtain \eqref{4.5}.

Lastly, to prove \eqref{4.6} we use \eqref{2.9} with $n=2m+2$. Then 
\[
\det_{0\leq i,j\leq 2m+1}\big(e_{i+j}^{+}\big)
= \det_{0\leq i,j\leq m}\big(e_{2(i+j)}^{+}\big)
\cdot\det_{0\leq i,j\leq m}\big(e_{2(i+j)+2}^{+}\big),
\]
and once again using \eqref{4.3} and \eqref{2.12} we get
\begin{equation}\label{4.10}
H_{2m+1}(e_{k}^{+}) = 4^{m+1}H_m\big(E_{2\mu}(\tfrac{q+r-s}{2q})\big)
\cdot H_{m}\big(E_{2\mu+2}(\tfrac{q+r-s}{2q})\big).
\end{equation}
We substitute \eqref{2.15} and \eqref{2.17} into \eqref{4.10}, both with $n=m$ 
and $x=(r-s)/q$. After some tedious but straightforward manipulations we get
\eqref{4.6}.
\end{proof}

As a first consequence of Theorem~\ref{thm:4.3} we consider a few more cases
of the generalized Bernoulli numbers and polynomials defined in \eqref{6.10}.
In particular, we will deal with certain Dirichlet characters modulo 8 and 12, 
as given in Table~2.

\bigskip
\begin{center}
{\renewcommand{\arraystretch}{1.2}
\begin{tabular}{|c|rrrr|}
\hline
$n$ & 1 & 3 & 5 & 7 \\
\hline
$\chi_{8,1}(n)$ & 1 & -1 & -1 & 1 \\
$\chi_{8,2}(n)$ & 1 & 1 & -1 & -1 \\
\hline
\end{tabular}}\quad
{\renewcommand{\arraystretch}{1.2}
\begin{tabular}{|c|rrrr|}
\hline
$n$ & 1 & 5 & 7 & 11 \\
\hline
$\chi_{12,1}(n)$ & 1 & -1 & -1 & 1 \\
$\chi_{12,2}(n)$ & 1 & 1 & -1 & -1 \\
\hline
\end{tabular}}

\medskip
{\bf Table~2}: Some characters modulo 8 and 12.
\end{center}
\bigskip

We note that both characters modulo 8 are primitive, and while $\chi_{12,1}$ is
also primitive, $\chi_{12,2}$ is induced from the character $\chi_3$ as defined
in Section~3 and therefore has conductor 3. However, this does not affect the
result that follows.

We begin with the character $\chi_{8,1}$. By \eqref{6.11}, \eqref{2.2}, and the
definition \eqref{4.1} we have
\begin{align*}
B_{n,\chi_{8,1}}(x) &= 8^{n-1}\left(B_n(\tfrac{x+1}{8})-B_n(\tfrac{x+3}{8})
-B_n(\tfrac{x+5}{8})+B_n(\tfrac{x+7}{8})\right) \\
&=8^{n-1}\frac{n}{2^n}\left(-E_{n-1}(\tfrac{x+1}{4})+E_{n-1}(\tfrac{x+3}{4})\right) \\
&= -\frac{n}{2}\cdot 4^{n-1}\cdot e_{n-1}^{-}(4,1,3;x).
\end{align*}
In the same way we can determine expressions for the remaining three cases. 
Upon setting $n=k+1$ we summarize the four cases as follows: For $q=4$ and 6
we have 
\begin{align}
b_k^{(1)}:=\frac{1}{k+1}B_{k+1,\chi_{2q,1}}(x) 
&= -\frac{1}{2}q^k e_k^{-}(q,1,q-1;x),\label{4.11}\\
b_k^{(2)}:=\frac{1}{k+1}B_{k+1,\chi_{2q,2}}(x) 
&= -\frac{1}{2}q^k e_k^{+}(q,1,q-1;x).\label
{4.12}
\end{align}
Applying the first part of Theorem~\ref{thm:4.3} to \eqref{4.11} and the second
part to \eqref{4.12} and using the identities \eqref{2.12} and \eqref{2.5}, we
get the following result.

\begin{corollary}\label{cor:4.4}
Let $q=4$ or $6$, set $\widetilde{q}:=(q-2)/q$, and let $b_k^{(1)}, b_k^{(2)}$
be as above. Then for all integers $m\geq 0$ we have 
$H_{2m}\big(b_k^{(1)}\big)=0$, and
\begin{align*}
%H_{2m}\big(b_k^{(1)}\big) &= 0,\\
H_{2m+1}\big(b_k^{(1)}\big)&=(-1)^{m+1}\left(\frac{q-2}{2}q^{2m}\right)^{2m+2}
\prod_{\ell=1}^m\left(\frac{\ell^2}{4}\left({\widetilde{q}}^2-(2\ell)^2\right)\right)^{2(m+1-\ell)},\\
%\end{align*}
%and furthermore,
%\begin{align*}
H_{2m}\big(b_k^{(2)}\big) &=(-1)^{m+1}\frac{q^{2m(2m+1)}}{m!^2}
\prod_{\ell=1}^m\left(\frac{\ell^2}{4}\left(\widetilde{q}^2-(2\ell-1)^2\right)\right)^{2(m+1-\ell)}, \\
H_{2m+1}\big(b_k^{(2)}\big) &=\left(\frac{1}{2}q^{2m+1}\right)^{2m+2}
\prod_{\ell=0}^m\frac{\ell!^4}{16^\ell}
\left(\widetilde{q}^2-(2\ell+1)^2\right)^{2(m-\ell)+1}.
\end{align*}
\end{corollary}

As another consequence of Theorem~\ref{thm:4.3} we consider the alternating analogue
of Corollary~\ref{cor:6.7}. For integers $s\geq 1$ and $k\geq 0$ we denote
\begin{equation}\label{4.13}
T_k(s) := 1-2^k+3^k-\cdots+(-1)^{s-1}s^k.
\end{equation}
There is a well-known connection with Euler polynomials, given by
\[
T_k(s) = \frac{1}{2}\big(E_k(1)-(-1)^sE_k(s+1)\big);
\]
see, e.g., \cite[Eq.~24.4.10]{DLMF}. With \eqref{4.1} this means that
\[
T_k(s) = \begin{cases}
\tfrac{1}{2}e_k^{-}(1,0,s;1) &\hbox{if $s$ is even},\\
\tfrac{1}{2}e_k^{+}(1,0,s;1) &\hbox{if $s$ is odd}.
\end{cases}
\]
Using \eqref{2.12} and Theorem~\ref{thm:4.3}, we immediately get the following
identities.

\begin{corollary}\label{cor:4.5}
Let $T_k(s)$ be as defined in \eqref{4.13}. 

$(a)$ When $s=2t$ is even, then for all integers $m\geq 0$ we have 
$H_{2m}(T_k(2t))=0$ and
\[
H_{2m+1}(T_k(2t))=\left(-t^2\right)^{m+1}\prod_{\ell=1}^m
\big(\ell^2(t^2-\ell^2)\big)^{2(m+1-\ell)},
\]
and in particular, $H_{2m+1}(T_k(2t))=0$ for $m\geq t$.

$(b)$ When $s$ is odd, then for all $m\geq 0$ we have  
\begin{align*}
H_{2m}(T_k(s))&=\frac{(-1)^m}{m!^2}\prod_{\ell=1}^m
\left(\frac{\ell^2}{4}\big(s^2-(2\ell-1)^2\big)\right)^{2(m+1-\ell)},\\
H_{2m+1}(T_k(s))&=\frac{1}{4^{m+1}}\prod_{\ell=0}^m
\frac{\ell!^4}{16^\ell}\big(s^2-(2\ell+1)^2\big)^{2(m-\ell)+1},
\end{align*}
and these determinants become $0$ when $m\geq(s+1)/2$, resp.\ $m\geq(s-1)/2$.
\end{corollary}

This result shows again that under certain circumstances all Hankel determinants
from a certain index on can vanish. In this connection it would be easy to 
state an ``Euler analogue" to Corollary~\ref{cor:6.6}. We leave this to the
interested reader. 

\section{Derivative sequences}

If for a sequence $(c_0, c_1,\ldots)$ we know the Hankel determinant $H_n(c_k)$,
then by \eqref{2.5} and \eqref{2.12} we also know $H_n(a\cdot b^k\cdot c_k)$ for
any numbers or variables $a$ and $b$. However, this is generally not the case
for $H_n(k\cdot c_k)$ or $H_n((k+1)\cdot c_k)$, or other expressions of this
kind. It is the purpose of this section to present a method that allows us to
deal with such expressions in some special cases.

We recall that both the Bernoulli and Euler polynomial sequences are {\it Appell
sequences\/}, that is, they satisfy the derivative property
\begin{equation}\label{5.1}
B_n'(x) = nB_{n-1}(x),\qquad E_n'(x) = nE_{n-1}(x).
\end{equation}
These identities follow quite easily form the generating functions in 
\eqref{1.1} and \eqref{1.2}, or from the identities in \eqref{2.1}. This 
gives rise to the question whether Hankel determinants of sequences or 
subsequences of Bernoulli, Euler, or generally Appell polynomials might give
rise to Hankel determinants of their derivatives. In general, this would be
asking too much; however, under certain circumstances we can indeed pass from
a polynomial sequence to its derivative, as the following theorem shows. We 
will prove it later in this section.

\begin{theorem}\label{thm:5.1}
Let $A_k(x)$, $k\geq 0$, be a sequence of $C^1$ functions and let
$x_0\in{\mathbb C}$ be such that $A_k(x_0)=0$ for all $k\geq 0$. Then
\begin{equation}\label{5.2}
H_n\big(A_k'(x_0)\big) = A_0'(x_0)^{n+1}
\lim_{x\to x_0}\frac{H_n\big(A_k(x_0)\big)}{A_0(x_0)^{n+1}}.
\end{equation}
\end{theorem}

For this result to be useful we need, above all, a sequence of functions whose
elements all have a root in common. But this is exactly the case for certain 
subsequences of Bernoulli and Euler polynomials, as one can see in \eqref{2.4a}
and \eqref{2.4b}. We use this fact in the following corollaries.

\begin{corollary}\label{cor:5.2}
For all $n\geq 0$ we have
\begin{equation}\label{5.3}
H_n\big((2k+1)E_{2k}\big) = 2^{2n(n+1)}\prod_{\ell=1}^n\ell!^4.
\end{equation}
\end{corollary}

\begin{proof}
We set $A_k(x):=E_{2k+1}(\frac{1+x}{2})$. Then $A_k(0)=E_{2k+1}(\frac{1}{2})=0$
for all $k\geq 0$ and by \eqref{5.1}, 
\[
A_k'(0) = \frac{2k+1}{2}E_{2k}(\tfrac{1}{2}) = \frac{2k+1}{2^{2k+1}}E_{2k},
\]
where we have also used the right-hand identity in \eqref{2.1}. 
Now, by \eqref{2.12} and \eqref{2.5} we have
\[
H_n\big((2k+1)E_{2k}\big) = H_n\big(2^{2k+1}A_k'(0)\big) 
= 2^{n+1}4^{n(n+1)}H_n\big(A_k'(0)\big).
\]
Next, since $A_0(x)=E_1(\frac{1+x}{2})=\frac{x}{2}$ and $A_0'(0)=\frac{1}{2}$, 
we get with \eqref{2.16} and \eqref{5.2},
\begin{align*}
H_n\big((2k+1)E_{2k}\big) &= 2^{n+1}4^{n(n+1)}\left(\tfrac{1}{2}\right)^{n+1}
(-1)^{\binom{n+1}{2}}\prod_{\ell=1}^n\left(\frac{\ell^2}{4}\big(0-(2\ell)^2\big)\right)^{n+1-\ell} \\
&= 4^{n(n+1)}\prod_{\ell=1}^n\big(\ell^4\big)^{n+1-\ell},
\end{align*}
and this is easily seen to be equivalent to \eqref{5.3}.
\end{proof}

The identity \eqref{5.3} can also be obtained by two alternative means: First,
we can use \eqref{2.3} and \eqref{2.13} with $x=1/4$, again applying 
\eqref{2.12} and \eqref{2.5}. And second, Corollary~\ref{cor:6.4} shows that
$H_n\big(kE_{k-1}\big)$ is of ``checkerboard type"; this means that we can use 
Lemma~\ref{lem:2.2} combined with \eqref{6.8}, and easily obtain \eqref{5.3}
again.

As a second application of Theorem~\ref{thm:5.1} we follow along the same lines
as in the proof of Corollary~\ref{cor:5.2}.

\begin{corollary}\label{cor:5.3}
For all $n\geq 0$ we have
\begin{equation}\label{5.4}
H_n\big((2^{2k+2}-1)B_{2k+2}\big)=\frac{(n+1)!}{2^{n+1}}\prod_{\ell=1}^n\ell!^4.
\end{equation}
\end{corollary}

\begin{proof}
Here we set $A_k(x):=E_{2k+2}(\frac{1+x}{2})$ and $x_0=1$. Then for all 
$k\geq 0$ we have $A_k(1)=E_{2k+2}(1)=0$ and also
\[
A_k'(1) = (k+1)E_{2k+1}(1) = \big(2^{2k+2}-1\big)B_{2k+2},
\]
where we have used the second identity in \eqref{2.3}. We also have
\[
A_0'(1) = 3B_2=\frac{1}{2},\qquad A_0(x)=E_2(\tfrac{1+x}{2})=\frac{x^2-1}{4}.
\]
Substituting everything, including \eqref{2.17}, into \eqref{5.2}, we get
\begin{align*}
H_n\big((2^{2k+2}-1)B_{2k+2}\big)
&= \left(\tfrac{1}{2}\right)^{n+1}(-1)^{\binom{n+1}{2}}
\prod_{\ell=1}^n\left(\frac{\ell^2}{4}\big(1-(2\ell+1)^2\big)\right)^{n+1-\ell} \\
&= \frac{1}{2^{n+1}}\prod_{\ell=1}^n\big(\ell^3(\ell+1)\big)^{n+1-\ell}.
\end{align*}
Finally, a straightforward manipulation shows that this is equivalent to 
\eqref{5.4}.
\end{proof}

\begin{corollary}\label{cor:5.4}
For all $n\geq 0$ we have
\begin{equation}\label{5.5}
H_n\big((2k+1)B_{2k}(\tfrac{1}{2})\big)
= \prod_{\ell=1}^n\frac{\ell!^8}{(2\ell)!(2\ell+1)!}.
\end{equation}
\end{corollary}

\begin{proof}
We take $A_k(x):=B_{2k+1}(\frac{1+x}{2})$ and $x_0=0$. Then 
$A_k(0)=B_{2k+1}(\frac{1}{2})=0$ for all $k\geq 0$. Furthermore,
\[
A_k'(0)=\frac{2k+1}{2}B_{2k}(\tfrac{1}{2}),\qquad
A_0'(0)=\frac{1}{2}B_0(\tfrac{1}{2})=\frac{1}{2},\qquad
A_0(x)=B_1(\tfrac{1+x}{2})=\frac{x}{2}.
\]
With \eqref{2.12} we now get
\[
H_n\big((2k+1)B_{2k}(\tfrac{1}{2})\big) = H_n\big(2A_k'(0)\big)
= 2^{n+1}H_n\big(A_k'(0)\big),
\]
and then with \eqref{5.2} and \eqref{2.13},
\begin{align*}
H_n\big((2k+1)B_{2k}(\tfrac{1}{2})\big) &= 2^{n+1}\left(\tfrac{1}{2}\right)^{n+1}
(-1)^{\binom{n+1}{2}}\prod_{\ell=1}^{n}
\left(\frac{\ell^{4}(0-\ell^{2})}{4(2\ell+1)(2\ell-1)}\right)^{n+1-\ell}\\
&= \prod_{\ell=1}^{n}\left(\frac{\ell^6}{4(2\ell+1)(2\ell-1)}\right)^{n+1-\ell}.
\end{align*}
Once again, a straightforward manipulation shows that this is equivalent to
\eqref{5.5}.
\end{proof}

To prove Theorem~\ref{thm:5.1} and to derive some further consequences, we need
some basics from the classical theory of orthogonal polynomials. Suppose we are
given a sequence ${\bf c}=(c_0, c_1, \ldots)$; then under certain conditions
there exists a positive Borel measure $\mu$ on ${\mathbb R}$ with infinite 
support such that 
\begin{equation}\label{5.6} 
c_k = \int_{\mathbb R}y^kd\mu(y),\qquad k = 0, 1, 2, \ldots
\end{equation}
We summarize several well-known facts and state them as a lemma, with a few
consequences; see, e.g., \cite[Ch.~2]{Is}, or \cite[Sect.~3]{DJ} for a somewhat
extended summary.

\begin{lemma}\label{lem:5.5}
If $\mu$ is the measure in \eqref{5.6}, there exists a unique sequence of monic
polynomials $P_n(y)$ of degree $n$, $n=0, 1, \ldots$, and a sequence of positive
numbers $(\zeta_n)_{n\geq 0}$, with $\zeta_0=1$, such that 
\begin{equation}\label{5.7}
\int_{\mathbb R}P_m(y)P_n(y)d\mu(y) = \zeta_n\delta_{m,n},
\end{equation}
where $\delta_{m,n}$ is the Kronecker delta function. Furthermore, for all 
$n\geq 1$ we have $\zeta_n=H_n({\bf c})/H_{n-1}({\bf c})$, and for $n\geq 0$,
\begin{equation}\label{5.8}
P_n(y) = \frac{1}{H_{n-1}({\bf c})}\det
\begin{pmatrix}
c_{0} & c_{1} & \cdots & c_{n}\\
c_{1} & c_{2} & \cdots & c_{n+1}\\
\vdots & \vdots & \ddots & \vdots\\
c_{n-1} & c_{n} & \cdots & c_{2n-1} \\
1 & y & \cdots & y^n
\end{pmatrix},
\end{equation}
where the polynomials $P_n(y)$ satisfy the $3$-term recurrence relation 
$P_0(y)=1$, $P_1(y)=y+s_0$, and
\begin{equation}\label{5.9}
P_{n+1}(y) = (y+s_n)P_n(y) - t_nP_{n-1}(y)\qquad (n\geq 1),
\end{equation}
for some sequences $(s_n)_{n\geq 0}$ and $(t_n)_{n\geq 1}$.
\end{lemma}

We continue with a couple of important consequences, summarized as a second
lemma.

\begin{lemma}\label{lem:5.6}
With the sequence $(c_k)$ and the polynomials $P_n(y)$ as in 
Lemma~\ref{lem:5.5}, we have for $0\leq r\leq n-1$
\begin{equation}\label{5.10}
y^rP_n(y)\bigg|_{y^k=c_k} = 0.
\end{equation}
Furthermore, with the sequence $(t_n)$ as in \eqref{5.9}, we have 
\begin{equation}\label{5.11}
H_n({\bf c}) = c_0^{n+1}t_1^nt_2^{n-1}\cdots t_{n-1}^2t_n\qquad (n\geq 0).
\end{equation}
\end{lemma}

There is also an interesting and important connection with certain continued
fractions ($J$-fractions in this case). However, this will not be needed here;
it can be found in various relevant publication, for instance, in
\cite[p.~20]{Kr}, or \cite[Sect.~3]{DJ}.

We are now ready to prove Theorem~\ref{thm:5.1}. Given a sequence of $C^1$ 
functions $A_k(x)$, by Lemma~\ref{lem:5.6} there is a sequence $P_n(y;x)$ of
monic orthogonal polynomials satisfying
\begin{equation}\label{5.12}
y^rP_n(y;x)\bigg|_{y^k=A_k(x)} = 0\qquad (0\leq r\leq n-1).
\end{equation}
This polynomial sequence $P_n(y;x)$ is sometimes called the
{\it monic orthogonal polynomials with respect to} $A_k(x)$.
We begin by proving the following key property.

\begin{lemma}\label{lem:5.7}
Let $A_k(x)$ be a sequence of $C^1$ functions, and let
$P_n(y;x)$ be the corresponding monic orthogonal polynomials. If $A_k(x_0)=0$
for some $x_0\in{\mathbb C}$ and for all $k\geq 0$, then $P_n(y;x_0)$ are the 
monic orthogonal polynomials with respect to the sequence $A_k'(x_0)$.
\end{lemma}

\begin{proof}
If we write
\[
P_n(y;x) = \sum_{j=0}^n\alpha_{n,j}(x)y^j,
\]
then by \eqref{5.12} we have for $0\leq r\leq n-1$,
\begin{equation}\label{5.13}
\sum_{j=0}^n\alpha_{n,j}(x)A_{j+r}(x) = 0.
\end{equation}
Since $A_k(x_0)=0$ for all $k\geq 0$, we have
\[
\sum_{j=0}^n\alpha_{n,j}(x)A_{j+r}(x_0) = 0.
\]
Subtracting this from \eqref{5.13} and dividing by $x-x_0$, we get
\[
\sum_{j=0}^n\alpha_{n,j}(x)\frac{A_{j+r}(x)-A_{j+r}(x_0)}{x-x_0} = 0.
\]
Finally, taking the limit as $x\to x_0$, we get
\[
\sum_{j=0}^n\alpha_{n,j}(x_0)A_{j+r}'(x_0) = 0;
\]
this, with \eqref{5.12}, proves the lemma.
\end{proof}

\begin{proof}[Proof of Theorem~\ref{thm:5.1}]
By Lemma~\ref{lem:5.7}, the sequences $A_k(x_0)$ and $A_k'(x_0)$ share the same
monic orthogonal polynomial. This means, in particular, that the terms 
$t_1, t_2, \ldots$ in \eqref{5.9} are the same, and therefore, by \eqref{5.11}
we have
\[
\lim_{x\to x_0}\frac{H_n\big(A_k(x_0)\big)}{A_0(x_0)^{n+1}}
=\frac{H_n\big(A_k'(x_0)\big)}{A_0'(x_0)^{n+1}}.
\]
This immediately leads to \eqref{5.2}, and the proof is complete.
\end{proof}

\section{Shifted sequences}

In the previous section we used the facts that 
$E_{2k+1}(\tfrac{1}{2})=E_{2k+2}(1)=B_{2k+1}(\tfrac{1}{2})$$=0$ for all $k\geq 0$
to obtain the identities \eqref{5.3}--\eqref{5.5}, respectively. We did that by
applying Theorem~\ref{thm:5.1} and using Theorems~\ref{thm:2.3} 
and~\ref{thm:2.4}. Apart from equivalent forms, there is one more sequence with
a common root we have not yet exploited, namely $B_{2k+1}(0)=0$. The problem
here is that this holds only for $k\geq 1$ since $B_1(x)=x+\tfrac{1}{2}$. 
Therefore
we cannot simply combine Theorem~\ref{thm:2.3} with Theorem~\ref{thm:5.1}, as
we did in the proof of Corollary~\ref{5.4}.

One possibility would be to consider $B_{2k+3}(x)$, which does indeed vanish 
for $x=0$ and for all $k\geq 0$. But we still have the problem that there is
no analogue of Theorem~\ref{thm:2.3} for the shifted sequence 
$B_{2k+3}(\tfrac{1+x}{2})$; however, this can be resolved as follows.

Given a sequence ${\bf c}=(c_0, c_1,\ldots)$, let $P_n(y)$, $n=0,1,\ldots$, be
the monic polynomials orthogonal with respect to ${\bf c}$, as in 
Lemmas~\ref{lem:5.5} and~\ref{lem:5.6}. With the coefficients $s_n$ and $t_n$
as in \eqref{5.9}, for all $n\geq 0$ we consider the determinant
\begin{equation}\label{6a.1}
d_{n}:=\det\begin{pmatrix}
-s_{0} & 1 & 0 & \cdots & 0\\
t_{1} & -s_{1} & 1 & \cdots & 0\\
0 & t_{2} & -s_{2} & \cdots & 0\\
\vdots & \vdots &  \ddots & \ddots & \vdots\\
0 & \cdots & 0 & t_{n} & -s_{n}
\end{pmatrix};
\end{equation}
thus, in particular, $d_0=-s_0$. These determinants play an important role in
connecting the Hankel determinant of a given sequence with that of a shifted
sequence.

\begin{lemma}[{\cite[Prop.~1.2]{MWY}}]\label{lem:6a.1}
With notations as above, we have, for a given sequence $(c_0, c_1,\ldots)$,
\begin{equation}\label{6a.2}
H_{n}(c_{k+1}) = d_{n}\cdot H_{n}(c_{k}).
\end{equation}
\end{lemma}

We now consider the special case $c_k=B_{2k+1}(\tfrac{1+x}{2})$. Then the
coefficients $s_n$, $t_n$ will be functions of $x$ and therefore $d_n=d_n(x)$
will also be a function of $x$. In fact, in \cite[Theorem~4.1]{DJ} we showed 
that
\begin{equation}\label{6a.5}
s_{n}=\binom{n+1}{2}-\frac{x^{2}-1}{4},\qquad
t_{n}=\frac{n^{4}(n^2-x^{2})}{4(2n+1)(2n-1)}.
\end{equation}
We can now state and prove the following result.

\begin{lemma}\label{lem:6a.2}
Consider the sequence $c_k=B_{2k+1}(\tfrac{1+x}{2})$ and let $d_n(x)$ be 
defined as in \eqref{6a.1}, which depends on $x$ due to \eqref{6a.5}.
Then for all $n\geq 2$ we have
\begin{equation}\label{6a.3}
\lim_{x\to -1}\frac{d_{n}(x)}{x^{2}-1} 
= \frac{(-1)^{n}2^{n-2}}{3^{n}}\prod_{\ell=2}^n
\left(\frac{(\ell+1)^{2}(2\ell-1)}{\ell(\ell-1)(2\ell+1)}\right)^{n+1-\ell}.
\end{equation}
\end{lemma}

\begin{proof} Using elementary determinant operations, we get from 
\eqref{6a.1} the recurrence relation 
\begin{equation}\label{6a.4}
d_{n+1}(x)=-s_{n+1}d_{n}(x)-t_{n+1}d_{n-1}(x).
\end{equation}
We prove \eqref{6a.3} by induction on $n$. By direct computation, using 
\eqref{6a.5} and \eqref{6a.1}, we obtain
\begin{align*}
\lim_{x\to -1}\frac{d_{2}(x)}{x^{2}-1} &=\frac{3}{10}
=\frac{(-1)^{2}2^{2-2}}{3^{2}}\cdot\frac{(2+1)^{2}(4-1)}{2(2-1)(4+1)},\\
\lim_{x\to -1}\frac{d_{3}(x)}{x^{2}-1} &=-\frac{36}{35}
=\frac{(-1)^{3}2^{3-2}}{3^{3}}\left(\frac{(2+1)^2(4-1)}{2(2-1)(4+1)}\right)^2
\frac{(3+1)^{2}(6-1)}{3(3-1)(6+1)},
\end{align*}
which is the induction beginning. Suppose now that \eqref{6a.3} is true up to
some $n$. Then we divide both sides of \eqref{6a.4} by $x^2-1$ and use the
induction hypothesis \eqref{6a.3} along with \eqref{6a.5}. After some
straightforward but tedious manipulations we obtain an expression for 
$\lim_{x\to -1}d_{n+1}(x)/(x^2-1)$, which is the same as the right-hand side of
\eqref{6a.3}, but with $n$ replaced by $n+1$. This completes the proof by
induction.
\end{proof}

We are now ready to prove the desired fourth consequence of 
Theorem~\ref{thm:5.1}.

\begin{corollary}\label{cor:6a.3}
For all $n\geq 0$ we have
\begin{equation}\label{6a.6}
H_n\big((2k+3)B_{2k+2}\big) = \frac{1}{2^{n+1}}\prod_{\ell=1}^n
\left(\frac{\ell^{3}(\ell+1)^{3}}{4(2\ell+1)^{2}}\right)^{n+1-\ell}.
\end{equation}
\end{corollary}

\begin{proof}
For $n=0$ and 1, the identity \eqref{6a.6} is easy to verify by direct 
calculation; we may therefore assume that $n\geq 2$.
We set $A_k(x):=B_{2k+3}(\frac{1+x}{2})$ and $x_0=-1$. Then we have 
$A_k(-1)=B_{2k+3}(0)=0$ for all $k\geq 0$, and also
\[
A_k'(x)=\frac{2k+3}{2}B_{2k+2}(\tfrac{1+x}{2}),\qquad
A_k'(-1)=\frac{2k+3}{2}B_{2k+2},
\]
as well as (see Table~1)
\[
A_0'(-1)=\frac{3}{2}B_2=\frac{1}{4},\qquad
A_0(x)=B_3(\tfrac{1+x}{2})=\frac{x(x^2-1)}{8}.
\]
With \eqref{2.12} and \eqref{5.2} we therefore get
\begin{align}
H_n\big((2k+3)B_{2k+2}\big) 
&= H_n\big(2A_k'(-1)\big) = 2^{n+1}H_n\big(A_k'(-1)\big)\label{6a.7} \\
&= 2^{n+1}\left(\tfrac{1}{4}\right)^{n+1}
\lim_{x\to-1}\frac{H_n\big(B_{2k+3}(\frac{1+x}{2})\big)}
{\left(\frac{x(x^2-1)}{8}\right)^{n+1}}. \nonumber
\end{align}
Now by Lemma~\ref{lem:6a.1} we have
\begin{equation}\label{6a.8}
\frac{H_n\big(B_{2k+3}(\frac{1+x}{2})\big)}{\left(\frac{x(x^2-1)}{8}\right)^{n+1}}
=4\cdot\frac{H_n\big(B_{2k+1}(\frac{1+x}{2})\big)}
{\left(\frac{x}{2}\right)^{n+1}\left(\frac{x^2-1}{4}\right)^n}
\cdot\frac{d_n(x)}{x^2-1}.
\end{equation}
Next, from \eqref{2.13} we get
\[
\lim_{x\to-1}\frac{H_n\big(B_{2k+1}(\frac{1+x}{2})\big)}
{\left(\frac{x}{2}\right)^{n+1}\left(\frac{x^2-1}{4}\right)^n}
= (-1)^{\binom{n+1}{2}}\frac{1}{3^n}\prod_{\ell=2}^{n}
\left(\frac{\ell^{4}(1-\ell^{2})}{4(2\ell+1)(2\ell-1)}\right)^{n+1-\ell}.
\]
Finally, substituting this and \eqref{6a.3} into \eqref{6a.8}, and then
\eqref{6a.8} into \eqref{6a.7}, we obtain the desired identity \eqref{6a.3}
after some easy manipulations.
\end{proof}

\section{A collection of Hankel determinant formulas}

As indicated in the Introduction, Hankel determinants of Bernoulli and Euler
numbers have been studied for many years, and numerous identities were derived
by different authors, often with differing notations and in different but
equivalent forms. In this section we attempt to collect all the identities we
are aware of and present them, as far as possible, in a unified format.

First we recall that in writing $H_n(b_k)$ it is assumed that we use the
definition \eqref{1.3a} and \eqref{1.3} and that the sequence $(b_k)$ begins
with $k=0$. Next, there is the issue of the close connection between Bernoulli
and Euler polynomials, given by identities such as \eqref{2.2}. Thus we have,
for instance,
\begin{align}
B_{2k}(\tfrac{1}{2}) &= \left(2^{1-2k}-1\right)B_{2k},\label{7.1}\\
\left(2^{2k+2}-1\right)B_{2k+2} &= (k+1)E_{2k+1}(1),\label{7.2}\\
E_k(1)&=\frac{2}{k+1}\left(2^{k+1}-1\right)B_{k+1}\qquad(k\geq 1),\label{7.3}\\
(2k+1)E_{2k} &= 2^{4k+2}B_{2k+1}(\tfrac{3}{4}),\label{7.4}
\end{align}
where \eqref{7.1} can be found in \cite[Eq.~24.4.27]{DLMF}, and 
\eqref{7.2}--\eqref{7.4} come from \eqref{2.3}. The left-hand sides of 
\eqref{7.1}--\eqref{7.4} are included in the tables below since they could be
considered somewhat simpler than the right-hand sides. In the case of 
\eqref{7.2} it is not clear which side could be considered ``simpler", and in
fact, we have also included the right-hand side (multiplied by 2).

Finally, we need to be aware of the fact that the products that occur in 
all identities for Hankel determinants can usually be written in at least two
different forms. One could argue, for instance, that the identities in the
statements of Corollaries~\ref{cor:5.2}--\ref{cor:5.4} are simpler and therefore
preferable to the ones at the end of the corresponding proofs. However, due
to the close connection between Hankel determinants and orthogonal polynomials,
especially as given by \eqref{5.11}, it makes sense to use the latter forms
as standard format. The following identities may serve to easily pass from one
form to the other:
\begin{align}
\prod_{\ell=1}^n\ell! &= \prod_{\ell=1}^n\ell^{n+1-\ell};\label{7.5}\\
\prod_{\ell=1}^n(2\ell+\nu)!  &= \nu!^n\prod_{\ell=1}^n
\big((2\ell-1+\nu)(2\ell+\nu)\big)^{n+1-\ell},\quad\nu=0,1,2;\label{7.6}\\
\prod_{\ell=1}^n(4\ell+\nu)!  &= \nu!^n\prod_{\ell=1}^n
\big((4\ell-3+\nu)\cdots(4\ell+\nu)\big)^{n+1-\ell},\quad\nu=0,1,2,3.\label{7.7}
\end{align}
These identities, which actually hold in greater generality, can be verified without much difficulty.

We are now ready to list the identities for Hankel determinants, mostly given 
in a standard format and organized in a couple of tables. The references 
provided are not necessarily the first occurrences in the literature. 

\subsection{Identities with nonzero terms for all $n$}

Most identities have nonzero Hankel determinants for all positive integers $n$;
we present them in the format
\begin{equation}\label{7.8}
H_n(b_k) = (-1)^{\varepsilon(n)}\cdot a^{n+1}\cdot\prod_{\ell=1}^{n}b(\ell)^{n+1-\ell}.
\end{equation}
Here, the column for $\varepsilon(n)$ could be eliminated by incorporating a $-$ sign in $a$ or $b(\ell)$, as appropriate. However, we decided to make the sign pattern more explicit. 

\medskip
\begin{center}
\renewcommand*{\arraystretch}{2.5}
\begin{longtable}[h]{|c||c|c|c|c|}
%\begin{tabular}{|c||c|c|c|c|}
\hline
$b_k$ & $\varepsilon(n)$ & $a$ & $b(\ell)$ & Reference \\
\hline\hline
$B_k$ & $\binom{n+1}{2}$ & 1 & $\displaystyle{\frac{\ell^4}{4(2\ell+1)(2\ell-1)}}$ & \cite[(3.56)]{Kr} \\
\hline
$B_{k+1}$ & $\binom{n+2}{2}$ & $\frac{1}{2}$ & $\displaystyle{\frac{\ell^2(\ell+1)^2}{4(2\ell+1)^2}}$ & \cite[(3.57)]{Kr} \\
\hline
$B_{k+2}$ & $\binom{n+1}{2}$ & $\tfrac{1}{6}$ & $\displaystyle{\frac{\ell(\ell+1)^2(\ell+2)}{4(2\ell+1)(2\ell+3)}}$ & \cite[(2.38)]{Kr} \\
\hline
$B_{2k+2}$ & $0$ & $\tfrac{1}{6}$ & $\displaystyle{\frac{\ell^3(\ell+1)(2\ell-1)(2\ell+1)^3}{(4\ell-1)(4\ell+1)^2(4\ell+3)}}$ & \cite[(3.59)]{Kr} \\
\hline
$B_{2k+4}$ & $n+1$ & $\tfrac{1}{30}$ & $\displaystyle{\frac{\ell(\ell+1)^3(2\ell+1)^3(2\ell+3)}{(4\ell+1)(4\ell+3)^2(4\ell+5)}}$ & \cite[(3.60)]{Kr} \\
\hline
$B_{2k}(\tfrac{1}{2})$ & $0$ & $1$ & $\displaystyle{\frac{\ell^4(2\ell-1)^4}{(4\ell-3)(4\ell-1)^2(4\ell+1)}}$ & \cite[(41)]{Ch} \\
\hline
$(2^{2k+2}-1)B_{2k+2}$ & $0$ & $\frac{1}{2}$ & $\displaystyle{\ell^3(\ell+1)}$ & Cor.~\ref{cor:5.3} \\
\hline
$(2k+1)B_{2k}(\tfrac{1}{2})$ & $0$ & $1$ & $\displaystyle{\frac{\ell^6}{4(2\ell+1)(2\ell-1)}}$ & Cor.~\ref{cor:5.4} \\
\hline
$(2k+3)B_{2k+2}$ & $0$ & $\frac{1}{2}$ & $\displaystyle{\frac{\ell^3(\ell+1)^3}{4(2\ell+1)^2}}$ & Cor.~\ref{cor:6a.3} \\
\hline
$B_{2k+1}(\tfrac{x+1}{2})$ & $\binom{n+1}{2}$ & $\frac{x}{2}$ & $\displaystyle{\frac{\ell^4(x^2-\ell^2)}{4(2\ell+1)(2\ell-1)}}$ & \cite[Thm.~1.1]{DJ} \\
\hline
$E_k$ & $\binom{n+1}{2}$ & 1 & $\displaystyle{\ell^2}$ & \cite[(4.2)]{AC} \\
\hline
$E_k(x)$ & $\binom{n+1}{2}$ & 1 & $\displaystyle{\frac{\ell^2}{4}}$ & \cite[(5.2)]{AC} \\
\hline
$E_{k+1}(1)$ & $\binom{n+1}{2}$ & $\frac{1}{2}$ & $\displaystyle{\frac{\ell(\ell+1)}{4}}$ & \cite[(H4)]{Ha} \\
\hline
$E_{2k}$ & $0$ & $1$ & $\displaystyle{(2\ell-1)^2(2\ell)^2}$ & \cite[(3.52)]{Kr} \\
\hline
$E_{2k+1}(1)$ & $0$ & $\frac{1}{2}$ & $\displaystyle{\frac{\ell^2(2\ell-1)(2\ell+1)}{4}}$ & \cite[(4.56)]{Mi} \\
\hline
$E_{2k+2}$ & $n+1$ & $1$ & $\displaystyle{(2\ell)^2(2\ell+1)^2}$ & \cite[(3.53)]{Kr} \\
\hline
$E_{2k+3}(1)$ & $n+1$ & $\frac{1}{4}$ & $\displaystyle{\frac{\ell(\ell+1)(2\ell+1)^2}{4}}$ & \cite[(4.57)]{Mi}\\
\hline
$(2k+1)E_{2k}$ & $0$ & 1 & $\displaystyle{(2\ell)^4}$ & Cor.~\ref{cor:5.2} \\
\hline
$(2k+2)E_{2k+1}(1)$ & $0$ & 1 & $\displaystyle{\ell^3(\ell+1)}$ & Cor.~\ref{cor:5.3} \\
%\hline
%$\displaystyle{\frac{E_{k}(1)}{k!}}$ & $\binom{n+1}{2}$ & $1$ & $\displaystyle{\frac{1}{4(2\ell-1)^2}}$ & unproven \\
\hline
$\displaystyle{\frac{E_{k+1}(1)}{(k+1)!}}$ & $\binom{n+1}{2}$ & $\frac{1}{2}$ & $\displaystyle{\frac{1}{4(2\ell-1)(2\ell+1)}}$ & \cite[(H12)]{Ha} \\
\hline
$\displaystyle{\frac{E_{2k+1}(1)}{(2k+1)!}}$ & $0$ & $\frac{1}{2}$ & $\displaystyle{\frac{1}{16(4\ell-3)(4\ell-1)^2(4\ell+1)}}$ & \cite[(H13)]{Ha} \\
\hline
$\displaystyle{\frac{E_{2k+3}(1)}{(2k+3)!}}$ & $n+1$ & $\frac{1}{24}$ & $\displaystyle{\frac{1}{16(4\ell-1)(4\ell+1)^2(4\ell+3)}}$ & \cite[(H22)]{Ha} \\
\hline
$E_{2k}(\tfrac{x+1}{2})$ & $\binom{n+1}{2}$ & 1 & $\displaystyle{\frac{\ell^2}{4}\big(x^2-(2\ell-1)^2\big)}$ & \cite[Cor.~5.2]{DJ} \\
\hline
$E_{2k+1}(\tfrac{x+1}{2})$ & $\binom{n+1}{2}$ & $\frac{x}{2}$ & $\displaystyle{\frac{\ell^2}{4}\big(x^2-(2\ell)^2\big)}$ & \cite[Cor.~5.2]{DJ} \\
\hline
$E_{2k+2}(\tfrac{x+1}{2})$ & $\binom{n+1}{2}$ & $\frac{x^2-1}{4}$ & $\displaystyle{\frac{\ell^2}{4}\big(x^2-(2\ell+1)^2\big)}$ & \cite[Cor.~5.2]{DJ} \\
\hline
\end{longtable}
%\end{tabular}
\end{center}

\subsection{Identities with zero terms for all even $n$}

A second class of identities have zero Hankel determinants for all positive 
even integers, that is,
\begin{equation}\label{7.9}
H_{2m}(b_k)=0;
\end{equation}
we then present the nonzero terms in the format
\begin{equation}\label{7.10}
H_{2m+1}(b_k)
=(-1)^{m+1}\cdot a^{2(m+1)}\cdot\prod_{\ell=1}^{m}b(\ell)^{2(m+1-\ell)}.
\end{equation}

\medskip
\begin{center}
\renewcommand*{\arraystretch}{2.5}
\begin{longtable}[h]{|c||c|c|c|}
%\begin{tabular}{|c||c|c|c|c|}
\hline
$b_k$ & $a$ & $b(\ell)$ & Reference \\
\hline\hline
$E_{k+1}$ & $1$ & $\displaystyle{(2\ell)^2(2\ell+1)^2}$ & \cite[(H8)]{Ha} \\
\hline
$E_{k+2}(1)$ & $\frac{1}{4}$ & $\displaystyle{\frac{\ell(\ell+1)(2\ell+1)^2}{4}}$ & \cite[(H11)]{Ha}\footnote{it appears that the author intended the determinant to be $H_{2n+1}$.} \\
\hline
$\left(0,E_1(1),E_2(1),\ldots\right)$ & $\frac{1}{2}$ & $\displaystyle{\frac{\ell^2(2\ell-1)(2\ell+1)}{4}}$ & \cite[(H9)]{Ha} \\
\hline
$kE_{k-1}(x)$ & $1$ & $\displaystyle{\ell^4}$ & \eqref{6.7} \\
\hline
$kE_{k-1}$ & $1$ & $\displaystyle{(2\ell)^4}$ & \eqref{6.8} \\
\hline
$\displaystyle{\left(0,\frac{E_1(1)}{1!},\frac{E_2(1)}{2!},\ldots\right)}$ & $\frac{1}{2}$ & $\displaystyle{\frac{1}{16(4\ell-3)(4\ell-1)^2(4\ell+1)}}$ & \cite[(H15)]{Ha} \\
\hline
$\displaystyle{\frac{E_{k+2}(1)}{(k+2)!}}$ & $\frac{1}{24}$ & $\displaystyle{\frac{1}{16(4\ell-1)(4\ell+1)^2(4\ell+3)}}$ & \cite[(H14)]{Ha} \\
\hline
$B_k(\tfrac{x+r}{q})-B_k(\tfrac{x+s}{q})$ & $\frac{s-r}{q^{m+1}}$ & $\displaystyle{\frac{\ell^4\big((s-r)^2-(q\ell)^2\big)}{4(2\ell-1)(2\ell+1)}}$ & \eqref{6.4} \\
\hline
$E_k(\tfrac{x+r}{q})-E_k(\tfrac{x+s}{q})$ & $\frac{s-r}{q}$ & $\displaystyle{\frac{\ell^2\big((s-r)^2-(2q\ell)^2\big)}{4q^2}}$ & \eqref{4.4} \\
\hline
\end{longtable}
\end{center}

\subsection{Miscellaneous identities}

We now collect a number of identities that do not fit into Subsections~7.1 or 
7.2. We begin with a few that are, however, closely related to some identities
in the two tables above. The first of these identities was adapted from 
Andrews and Wimp \cite[p.~441]{AW}:
\begin{equation}\label{7.11a}
H_n\left(\frac{B_{k}}{k!}\right) = (-1)^{\binom{n+1}{2}}(n+1)
\prod_{\ell=1}^n\left(\frac{1}{4(2\ell-1)(2\ell+1)}\right)^{n+1-\ell}.
\end{equation}
The next three identities are due to Krattenthaler and were published in 
\cite[p.~346]{FY}.
\begin{align}
H_n\left(\frac{B_{2k+2}}{(2k+2)!}\right) &= \left(\frac{1}{4}\right)^{(n+1)^2}
\prod_{\ell=1}^{2n+1}\left(\frac{1}{2\ell+1}\right)^{2n+2-\ell},\label{7.11b}\\
H_n\left(\frac{B_{2k+4}}{(2k+4)!}\right) &= \left(\frac{-1}{36}\right)^{n+1}
\left(\frac{1}{4}\right)^{(n+1)^2}
\prod_{\ell=1}^{2n+1}\left(\frac{1}{2\ell+3}\right)^{2n+2-\ell},\label{7.11c}\\
H_n\left(\frac{B_{2k+6}}{(2k+6)!}\right) &= \frac{(n+2)(2n+5)}{3\cdot 60^{2n+2}}
\left(\frac{1}{4}\right)^{(n+1)^2}
\prod_{\ell=1}^{2n+1}\left(\frac{1}{2\ell+5}\right)^{2n+2-\ell}.\label{7.11d}
\end{align}
The identity \eqref{7.11d} was slightly changed from its original form. The
following three identities were adapted from (H21), (H23), and (H24), 
respectively, in~\cite{Ha}.
\begin{align}
H_n\left(\frac{E_{k+3}(1)}{(k+3)!}\right)
&= (-1)^{\binom{n+2}{2}}\left(\frac{1}{24}\right)^{n+1}\prod_{\ell=1}^n
\left(\frac{1}{4(2\ell+1)(2\ell+3)}\right)^{n+1-\ell}\label{7.11}\\
&\qquad\times\begin{cases}
\binom{n+3}{2},& (n\;\hbox{odd}),\\
\binom{n+2}{2},& (n\;\hbox{even});
\end{cases}\nonumber \\
H_n\left(\frac{E_{2k+5}(1)}{(2k+5)!}\right)
&= \left(\frac{1}{2\cdot 6!}\right)^{n+1}\binom{2n+4}{2}\label{7.12}\\
&\qquad\times\prod_{\ell=1}^n
\left(\frac{1}{16(4\ell+1)(4\ell+3)^2(4\ell+5)}\right)^{n+1-\ell};\nonumber\\
H_n\left(\frac{E_{2k+7}(1)}{(2k+7)!}\right)
&= \left(\frac{-1}{5\cdot 8!}\right)^{n+1}\frac{4n^2+18n+17}{3}\binom{2n+6}{4}\label{7.13}\\
&\qquad\times\prod_{\ell=1}^n
\left(\frac{1}{16(4\ell+3)(4\ell+5)^2(4\ell+7)}\right)^{n+1-\ell}.\nonumber
\end{align}
Another identity of a similar nature is
\begin{align}
H_n\big(B_{k+2}(-1)\big) &= (-1)^{\binom{n+1}{2}}\left(\frac{1}{6}\right)^{n+1}
\big((n+1)(n+2)^2(n+3)+1\big)\label{7.14} \\
&\qquad\times\prod_{\ell=1}^n\left(\frac{\ell(\ell+1)^2(\ell+2)}{4(2\ell+1)(2\ell+3)}\right)^{n+1-\ell},\nonumber
\end{align}
which can be found in \cite[Eq.~(7.2)]{FK}, in a slightly different form. In
this connection Fulmek and Krattenthaler also showed that 
\[
H_n\left(B_k-2B_{k+1}+B_{k+2}\right) = H_n\left(B_{k+2}(-1)\right)\qquad
(n\geq 0).
\]
Furthermore, they derived identities for
\[
H_{2m}\left(B_{k+2}(-\tfrac{1}{2})\right)\quad\hbox{and}\quad
H_{2m+1}\left(B_{k+2}(-\tfrac{1}{2})\right),
\]
involving $_4F_3$ hypergeometric functions \cite[Eq.~(7.3), (7.4)]{FK} and
identities for
\[
H_{2m}\left(B_{k+2}(\tfrac{1}{2})\right)\quad\hbox{and}\quad
H_{2m+1}\left(B_{k+2}(\tfrac{1}{2})\right),
\]
which involve certain finite sums \cite[Eq.~(7.5), (7.6)]{FK}. At this point
we also mention the identities for
\[
H_{2m}\left(E_k(\tfrac{x+r}{q})+E_k(\tfrac{x+s}{q})\right)\quad\hbox{and}\quad
H_{2m+1}\left(E_k(\tfrac{x+r}{q})+E_k(\tfrac{x+s}{q})\right),
\]
which were obtained in \eqref{4.5} and \eqref{4.6} above.

We conclude this list of identities with a very general formula, which is also
due to Fulmek and Krattenthaler \cite[Eq.~(5.3)]{FK}. Here a well-known symbolic
notation is used, where after expansion each power ${\mathcal B}^j$ is replaced
by the Bernoulli number $B_j$. Also, the shifted factorial $(a)_j$ is defined by
$(a)_j:=a(a+1)\cdots(a+j-1)$ for $j\geq 1$, and $(a)_0=1$. Thus, for example, we
have
\[
{\mathcal B}^3({\mathcal B}+1)_2={\mathcal B}^3({\mathcal B}+1)({\mathcal B}+2) 
= {\mathcal B}^5+3{\mathcal B}^4+2{\mathcal B}^3 \equiv B_5+3B_4+2B_3.
\]
We can now state the identity in question, which is, in fact, again of the form
\eqref{7.8}:

%\begin{theorem}[\cite{FK}, Theorem~23]\label{thm:7.1}
For integers $a,b\geq 1$ and $c,d\geq 0$, we have
\begin{align}
H_n\big(&{\mathcal B}^{k+2}({\mathcal B}+1)_{a-1}({\mathcal B}+1)_{b-1}(-{\mathcal B}+1)_{c-1}(-{\mathcal B}+1)_{d-1}\big)\label{7.15}\\
&=(-1)^{\binom{n+1}{2}}\left(\frac{(a+c-1)!(b+c-1)!(a+d-1)!(b+d-1)!}{(a+b+c+d-1)!}\right)^{n+1}\nonumber\\
&\times\prod_{\ell=1}^n\left(\frac{\ell(a+c+\ell-1)(b+c+\ell-1)(a+d+\ell-1)}
{(a+b+c+d+2\ell-3)(a+b+c+d+2\ell-2)^2}\right.\nonumber\\
&\qquad\quad\times\left.\frac{(b+d+\ell-1)(a+b+c+d+\ell-2)}{(a+b+c+d+2\ell-1)}\right)^{n+1-\ell},\nonumber
\end{align} 
where in the case $c=0$ or $d=0$ we interpret $(-{\mathcal B}+1)_{-1}$ as 
$1/(-{\mathcal B})$.
%\end{theorem}

As mentioned in \cite[p.~626]{FK}, the cases $a=b=1$, $c=d=0$ and $a=b=c=d=1$
give the first and third entries, respectively, in the table in Subsection~7.1. 
Similarly, $a=b=c=1$, $d=0$, would give the second entry in this table.

\subsection{Other related sequences}

We made the conscious decision to restrict ourselves to Bernoulli and Euler 
numbers and polynomials in Subsections~7.1--7.3. We conclude this section with
a few remarks on related sequences.

{\bf 1.} While number theorists and researchers in special functions tend to 
favor the definition \eqref{1.2} for Euler numbers, combinatorists typically 
prefer the alternative sequence ${\bf E}_n$ defined by
\begin{equation}\label{7.16}
\tan{t}+\sec{t} = \sum_{n=0}^{\infty}{\bf E}_n\frac{t^n}{n!},
\end{equation}
where we use a different font to avoid possible confusion. The first few terms,
starting with ${\bf E}_0$, are 1, 1, 1, 2, 5, 16, 61, 272, 1385; they are all
positive integers. By comparing the generating function \eqref{7.16} with 
\eqref{1.2}, it is not difficult to see that, for all $k\geq 0$,
\begin{align}
{\bf E}_{2k} &= (-1)^kE_{2k},\label{7.17}\\
{\bf E}_{2k+1} &= (-1)^k 2^{2k+1}E_{2k+1}(1).\label{7.18}
\end{align}
Using these identities and \eqref{2.12}, any Hankel determinant identity for
the sequences on the right immediately give identities for the ones on the left,
and vice versa; in fact, this is how we imported the numerous identities from
Han's recent paper \cite{Ha}. There are eight more identities in \cite{Ha} for
the ``mixed" sequence $({\bf E}_k)$, namely for
\[
H_n\big({\bf E}_{k+\mu}\big),\;\mu=0,1,2,\quad\hbox{and}\quad
H_n\big({\bf E}_{k+\nu}/(k+\nu)!\big),\;\nu=0,1,2,3,4.
\]

{\bf 2.} By \eqref{7.16} it is clear that the numbers ${\bf E}_{2k-1}$, 
$k\geq 1$, are the same as the {\it tangent numbers} (or {\it tangent
coefficients}) $T_k$, which are also known to have the form
\[
T_k = (-1)^{k-1}2^k\frac{2^{2k}-1}{2k}B_{2k}
\]
(see, e.g., \cite[Eq.~4.19.3]{DLMF}). By \eqref{7.18} we have
$T_k=(-1)^{k-1}2^{2k-1}E_{2k-1}(1)$; therefore all identities for $E_k(1)$ can
also be seen as identities for tangent numbers.

{\bf 3.} Euler numbers of an integer order $p\geq 1$ are defined by a
generating function that is the $p$th power of the left-hand identity in
\eqref{1.2}. Already Al-Salam and Carlitz \cite{AC} found the Hankel 
determinants of the sequence of these higher-order Euler numbers. More 
recently, Han \cite{Ha} dealt with other related sequences of higher-order
Euler numbers, and the second author and Shi \cite{JS} determined the 
orthogonal polynomials of higher-order Euler {\it polynomials\/}, which also
led to relevant Hankel determinants. Higher-order {\it Bernoulli} numbers,
however, are more challenging; see the remarks in \cite[p.~401]{JS}.

{\bf 4.} Among other generalizations of Bernoulli and Euler numbers for which
Hankel determinants have been computed are the $q$-Bernoulli-Carlitz numbers
\cite{CZ}, the median Bernoulli numbers \cite{Ch}, and some character
analogues (Corollaries~\ref{cor:6.5} and~\ref{cor:4.4} above). We refer the
interested reader to three very extensive studies by Krattenthaler
\cite{Kr, Kr2} and Milne \cite{Mi} for numerous other Hankel determinant
evaluations. Extensive references to the vast literature are provided in
\cite[pp.~47--48]{Kr}, \cite[p.~122]{Kr2}, and \cite[pp.~54--57]{Mi}.

\end{document}